\documentclass[a4paper]{amsart}
\usepackage{amssymb}
\usepackage{stmaryrd}
\usepackage[all]{xy}
\usepackage{epsf}
\usepackage{enumerate}
 
\newcommand{\Hrule}{\rule{\linewidth}{0.5mm}}

\newtheorem{theorem}{Theorem}[section]
\newtheorem{theoremz}{Theorem}             
\newtheorem{corollaryz}{Corollary} 

\newtheorem{proposition}[theorem]{Proposition}
\newtheorem{lemma}[theorem]{Lemma}
\newtheorem{corollary}[theorem]{Corollary}

\theoremstyle{definition}
\newtheorem{definition}[theorem]{Definition}
\newtheorem{example}[theorem]{Example}

\theoremstyle{remark}
\newtheorem{remark}[theorem]{Remark}

\newcommand{\Rr}{\mathbb R}

\newcommand{\set}[1]{\left\{#1\right\}}

\newcommand{\eps}{\varepsilon}

\newcommand{\rmap}{\longrightarrow}

\newcommand{\X}{\ensuremath{\mathfrak{X}}}

\newcommand{\G}{\mathcal{G}}            
\newcommand{\OO}{\mathcal{O}}            
\newcommand{\NN}{\mathcal{N}}            
\newcommand{\HH}{\mathcal{H}}            
\newcommand{\KK}{\mathcal{K}}            
\newcommand{\K}{\mathcal{K}}            
\newcommand{\al}{\alpha}                

\renewcommand{\d}{\mathrm{d}}

\begin{document}


\begin{titlepage}
\begin{center}
\Hrule \\[0.4cm]
{\huge \bfseries On the linearization theorem for proper Lie groupoids}\\[0.4cm]
\Hrule \\[1.5cm]

\large{Marius Crainic\footnote{Supported by NWO Vidi Project ÒPoisson topologyÓ -- grant 639.032.712} and Ivan Struchiner}\footnote{Supported by NWO Vidi Project ÒPoisson topologyÓ -- grant 639.032.712\\ \\



Mots cl\'es:  Groupo\"ides de Lie, actions propres, lin\'earisation, feuilletages, stabilit\'e de Reeb local}
\end{center}
\vspace{2.0cm}
\textbf{Abstract:} We revisit the linearization theorems for proper Lie groupoids around general orbits (statements and proofs). In the the fixed point case (known as Zung's theorem) we give a shorter and more geometric proof, based on a Moser deformation argument. The passage to general orbits (Weinstein) is given a more conceptual interpretation: as a manifestation of Morita invariance. We also clarify the precise statements of the Linearization Theorem (there has been some confusion on this, which has propagated throughout the existing literature). 

\vspace{2cm}
\begin{center}
\Hrule \\[0.4cm]
{\huge \bfseries Sur le th\'eor\'eme de lin\'earisation pour les groupo\"ides de Lie propres}\\[0.4cm]
\Hrule \\[1.5cm]
\end{center}
\textbf{R\'esum\'e:} Nous revisitons les th\'eor\`emes de lin\'earisation pour les
groupo\"ides de Lie propres autour des orbites g\'en\'erales. Dans le
cas du point fixe (connu sous le nom de th\'eor\`eme de Zung), nous
donnons une preuve plus courte et plus g\'eom\'etrique, bas\'ee sur
l'argument de deformation de Moser. Le passage au cas g\'en\'eral est
d\'ecrit de fa\c{c}on plus conceptuelle, comme manifestation de
l'invariance de Morita. Nous clarifions \'egalement les conditions
n\'ecessaires \`a l'application du th\'eor\`eme (il y a eu assez de
confusions dans ce qui concerne l\'enonc'e precis du th\'eor\`eme, et
qui se sont propag\'ees dans la litt\'erature.)

\end{titlepage}

\title[The linearization theorem for proper Lie groupoids]{On the linearization theorem for proper Lie groupoids}

\author{Marius Crainic}
\address{Depart. of Math., Utrecht University, 3508 TA Utrecht, 
The Netherlands}
\email{m.crainic@uu.nl}

\author{Ivan Struchiner}
\address{Depart. of Math., University of S\~ao Paulo, Rua do Mat\~ao 1010, 
S\~ao Paulo, SP, Brazil, CEP: 05508-090}
\email{ivanstru@ime.usp.br}

\subjclass[2000]{}

\maketitle

\setcounter{tocdepth}{1}

\section*{Introduction}             %

The linearization theorem for Lie groupoids is a far reaching generalization of the
tube theorem (for Lie group actions), Ehresmann's theorem (for proper submersions), and Reeb stability
(for foliations). It was first addressed by A. Weinstein 
as part of his program that aims at a geometric understanding of Conn's linearization
theorem in Poisson Geometry \cite{Wein2}. Various partial results have been obtained over the last 10 years (see Section \ref{sec: existing results}). However, even though it was a general belief that \emph{every proper Lie groupoid is linearizable}, such a statement has never been made precise or proven. Moreover, even for the existing results there has been some confusion regarding their precise statements (see also section \ref{sec: existing results} bellow); this confusion has propagated throughout  the existing literature.


The aim of this note is to clarify the statement of the linearization theorem and to present a simple geometric proof of it. 
Recall that, for any orbit $\OO$ of a Lie groupoid $\G$, there is a local model for $\G$ around $\OO$, the linearization $\NN_{\OO}(\G)$. 
The theorem refers to the equivalence of $\G$ and $\NN_{\OO}(\G)$ near $\OO$. The main theorem we discuss is the following.

\begin{theoremz}\label{main-theorem} Let $\G$ be a Lie groupoid over $M$, and $\mathcal{O}\subset M$ an orbit through $x\in M$.
If $\G$ is proper at $x$, then 
there are neighborhoods $U$ and $V$ of $\OO$ such that $\G|_{U}\cong \NN_{\OO}(\G)|_V$. 
\end{theoremz}

The first consequence is a stronger version of the result obtained by joining the works of Weinstein \cite{Wein3} and Zung \cite{Zung}:

\begin{corollaryz}\label{cor1} If $\G$ is proper at $x$ and $\OO$ is of finite type, 
then one can find arbitrarily small neighborhoods $U$ of $\OO$ in $M$ such that $\G|_{U}\cong \NN_{\OO}(\G)$.
\end{corollaryz}

Of course, requiring $U$ invariant is a natural condition. In this direction we have:

\begin{corollaryz}\label{cor2} If $\G$ is $s$-proper at $x$, then 
one can find arbitrarily small invariant neighborhood $U$ of $\OO$ in $M$ such that $\G|_{U}\cong \NN_{\OO}(\G)$.
\end{corollaryz}

Note that $s$-properness implies that $\OO$ is compact. But also the property ``one can find arbitrarily small invariant neighborhoods" alone is another strong 
property of $\OO$ (called stability). Actually, if $\G$ is proper at $x$, then the stability 
of $\OO$ forces $s$-properness at $x$. With the mind at proper actions by {\it non-compact} Lie groups,
we address the following problem (where the orbit may be non-compact). \\


\hspace*{-.2in}\textbf{Problem.} {\it If $\G$ is proper at $x$ and $s$ is trivial on an invariant neighborhood of $x$, under what extra-hypothesis on $\OO$ (if any), does it follow that one can find an invariant neighborhood $U$ of $\OO$ in $M$ such that $\G|_{U}\cong \NN_{\OO}(\G)$.}




\newpage

\section{A more detailed introduction}

\subsection{\underline{Lie groupoid notations; properness}}
Throughout this paper, $\G$ will denote a Lie groupoid over a manifold $M$. Hence
$M$ is the manifold of objects, $\G$ is the manifold of arrows (in this paper all manifolds are assumed to be Hausdorff and second countable). 
We denote by $s, t: \G\rmap M$ the source and the target maps and by $gh$ the multiplication (composition) of arrows 
$g, h\in \G$ (defined when $s(g)= t(h)$).

\begin{definition}\label{def-properness} Let $\G$ be a Lie groupoid over $M$, $x\in M$. We say that $\G$ is
\begin{itemize}
\item \textbf{proper} if the map $(s, t): \G\rmap M\times M$ is a proper map. 
\item \textbf{$s$-proper} if the map $s: \G\rmap M$ is proper.
\item \textbf{proper at $x$} if the map $(s, t)$ is proper at $(x,x)$, i.e., if any sequence $(g_n)_{n\geq 1}$ in $\G$
with $(s(g_n), t(g_n))\to (x, x)$ admits a convergent subsequence.
\item \textbf{$s$-proper at $x$} if the map $s$ is proper at $x$.
\end{itemize}
\end{definition}

\begin{remark}\label{remark-properness}\rm\ 
It follows from Proposition \ref{prop: morita invariance of properness} that properness at $x$ implies properness at any point in the orbit $\OO$ of $x$. Also, all the obvious implications above are strict. On the other hand, Ehresmann's theorem implies that, when the fibers of $s$ are connected (which happens in many examples!), 
$s$-properness at $x$ is equivalent to the compactness of $s^{-1}(x)$. A version of Ehresmann's theorem (``at $x$'') implies 
that $s$-properness at $x$ is actually equivalent to the condition that $s^{-1}(x)$ is compact and $s$ is trivial around $x$. 
\end{remark}

\begin{example}\label{ex-action}\rm \ 
The main example to have in mind is the Lie groupoid associated to the action of a Lie group $G$ on a manifold $M$. 
Known as the action Lie groupoid, and denoted $G\ltimes M$, it is a Lie groupoid over $M$ whose manifold of arrows 
is $G\times M$, with source/target defined by $s(g, x)= x$, $t(g, x)= gx$ 
and the multiplication $(g, x)(h, y)= (gh, y)$
(defined when $x= hy$). The action groupoid $G\ltimes M$ is proper (or proper at $x$) if and only if 
the action of $G$ on $M$ is proper (or proper at $x\in M$, see e.g. \cite{DuKo}); $s$-properness corresponds to the
compactness of $G$. 
\end{example}

\begin{example}\label{ex-submersion} \rm \ For any submersion $\pi: X\rmap Y$ one has a groupoid over $X$:
\[ \G(\pi)= X\times_{\pi}X= \{(x, x')\in X\times X: \pi(x)=\pi(x') \},\]
with $s(x, y)= y$, $t(x, y)= x$ and multiplication $(x, y)(y, z)= (x, z)$. While $\G(\pi)$ is always proper, it is
$s$-proper if and only if $\pi$ is a proper map. When $Y$ is a point, the resulting groupoid $X\times X$
is known as the pair groupoid over $X$ (always proper, and $s$-proper if and only if $X$ is compact).
\end{example}

\begin{example}\label{ex-principal-bundles} \rm \ Associated to any principal $G$-bundle $\pi: P\rmap M$
there is a Lie groupoid over $M$, known as the gauge groupoid of $P$, denoted $\textrm{Gauge}(P)$, which is the quotient $(P\times  P)/G$
of the pair groupoid of $P$ modulo the diagonal action of $G$. It
is proper if and only if $G$ is compact; it is $s$-proper if and only if $P$ is compact.
\end{example}

\subsection{\underline{Orbits and the local model}} Given a Lie groupoid $\G$ over $M$, two points $x, y\in M$ are in the same orbit of $\G$ if there is an arrow $g: x\rmap y$ (i.e. $s(g)= x$, $t(g)= y$).
This induces the partition of $M$ by the \textbf{orbits} of $\G$. Each orbit carries a canonical smooth structure that makes
it into an immersed submanifold of $M$ (cf. e.g. \cite{MoMr}, but see also below).

Let $\mathcal{O}$ be an orbit. The linearization theorem at $\mathcal{O}$ provides a ``linear'' model for $\G$ around $\mathcal{O}$. This model is
just the tubular neighborhood in the world of groupoids, for $\G$ near $\mathcal{O}$. More precisely, over $\mathcal{O}$, $\G$ restricts to a Lie groupoid
\[ \G_{\mathcal{O}}:= \{ g\in \G: s(g), t(g)\in \mathcal{O} \}.\]
Its normal bundle in $\G$ sits over the normal bundle of $\OO$ in $M$:
\begin{equation}
\xymatrix{
\mathcal{N}_{\mathcal{O}}(\G):= T\G/T\G_{\mathcal{O}} \ar[r]_-{dt} \ar@<+1ex>[r]^-{ds}&   \mathcal{N}_{\mathcal{O}}:= TM/T\mathcal{O}}
\end{equation}
as a Lie groupoid. The groupoid structure is induced from the groupoid structure of the tangent groupoid $T\G$ (a groupoid over $TM$); i.e. the structure maps (source, the target and the multiplication) are induced
by the differentials of those of $\G$.

There are various ways of realizing $\mathcal{N}_{\mathcal{O}}(\G)$
more concretely. For instance, since $\G_{\mathcal{O}}= s^{-1}(\mathcal{O})$, its normal bundle in $\G$ is just the pull-back of $\mathcal{N}_{\mathcal{O}}$ by $s$:
\[ \mathcal{N}_{\mathcal{O}}(\G)= \{(g, v)\in \G_{\mathcal{O}}\times \mathcal{N}_{\mathcal{O}}: s(g)= \pi(v)\}.\]
With this, the groupoid structure comes from a (fiberwise) linear action of $\G_{\mathcal{O}}$ on $\mathcal{N}_{\mathcal{O}}$. That means that any arrow
$g: x\rmap y$ in $\G_{\OO}$ induces a linear isomorphism
\begin{equation}
\label{isotropy-action}
g: \mathcal{N}_x\rmap \mathcal{N}_y .
\end{equation}
Explicitly: given $v\in \mathcal{N}_x$, one chooses a curve $g(t): x(t)\rmap y(t)$ in $\G$ with $g(0)= g$ and such that $\dot{x}(0)\in T_xM$ represents $v$, and then $gv$ is represented by $\dot{y}(0)\in T_yM$.
With these, the groupoid structure of $\mathcal{N}_{\mathcal{O}}(\G)$ is given by 
\[ s(g, v)= v, \ \  t(g, v)= gv, \  \ (g, v)(h, w)= (gh, w) .\]

\begin{definition} 
The Lie groupoid $\mathcal{N}_{\mathcal{O}}(\G)$ is called \textbf{the linearization of $\G$ at $\OO$}. 
\end{definition}

\begin{example} \rm \ When $x$ is a fixed point of $\G$, i.e. $\mathcal{O}_x= \{x\}$, then $\mathcal{N}_{\mathcal{O}_x}(\G)$
is the groupoid associated (cf. Example \ref{ex-action}) to the action (\ref{isotropy-action}) of $G_x$ on $\NN_x$. 
\end{example}

The Lie groupoid $\mathcal{N}_{\mathcal{O}}(\G)$ can be further unravelled by choosing a point $x\in \mathcal{O}$. 
The outcome is a bundle-description of the linearization, which is closer to  the familiar one from group actions. Here are the details.
We fix $x\in \mathcal{O}$ and we use the notation $\mathcal{O}= \mathcal{O}_x$. Associated to $x$ there are:

\begin{itemize}
\item the \textbf{$s$-fiber at $x$}, $P_x:= s^{-1}(x)$, which is a submanifold of $\G$.
\item the \textbf{isotropy group at $x$}, $G_x=s^{-1}(x)\cap t^{-1}(x)$ which is a Lie group with multiplication and  smooth structure induced from the ones of $\G$.
\item the \textbf{isotropy representation at $x$}, $\NN_x= T_xM/T_x\OO_x$, viewed as a representation of $G_x$ with the action (\ref{isotropy-action}) described above.
\item the \textbf{isotropy bundle at $x$}, which is $P_x$ viewed as a principal $G_x$-bundle with the action induced by the multiplication of $\G$
\begin{equation}
\label{orbit-bundle}  
t: P_x\rmap \mathcal{O}_x .
\end{equation}
\end{itemize}
It is this description that provides the orbit with its smooth structure:
since the action of $G_x$ on $P_x$ is free and proper, it induces a smooth structure on $\mathcal{O}_x$, making 
(\ref{orbit-bundle}) into a smooth bundle and $\mathcal{O}_x$ into an immersed submanifold of $M$. 
With these, 
\begin{itemize}
\item the normal bundle $\mathcal{N}_{\mathcal{O}}$ is isomorphic to the associated vector bundle
\[ \mathcal{N}_{\mathcal{O}}\cong P_x\times_{G_x} \mathcal{N}_x, \]
i.e., the quotient of $P_x\times \NN_x$ modulo the action $\gamma \cdot (g, v)= (g\gamma^{-1}, \gamma v)$ of $G_x$. 
\item similarly, the space of arrows of the linearization is
\[ \mathcal{N}_{\mathcal{O}}(\G) \cong (P_x\times P_x)\times_{G_x} \mathcal{N}_x. \]
\item In this new description of $\mathcal{N}_{\mathcal{O}}(\G)$, the groupoid structure is given by 
\[ s([p, q, v])= [q, v], t([p, q, v])= [p, v], [p, q,v]\cdot [q, r, v]= [p, r, v].  \]
\end{itemize}

\subsection{\underline{Linearizability}} Here we explain the meaning of ``linearizability''. 

\begin{definition}\label{def-invariance} Let $\G$ be a Lie groupoid over $M$. For any open set $U\subset M$,
\begin{itemize}
\item we denote by $\G|_{U}$ the restriction of $\G$ to $U$, which is the Lie groupoid consisting of arrows of $\G$ which start and end at points in $U$.
\item we say that $U$ is \textbf{invariant with respect to $\G$} (or simply $\G$-invariant) if it contains every
orbit that it meets. 
\end{itemize}
\end{definition}

\begin{example} \rm \ For the linear model, using the bundle picture,
open sets $U\subset \NN_{\OO}$ correspond to $G_x$-invariant open neighborhoods $P$ of $P_x\times \{0\}$ in $P_x\times \NN_x$ 
($U= P/G_x$); $U$ is invariant w.r.t. $\NN_{\OO}(\G)$ if and only if $P$ is a product neighborhood, i.e. of type $P_x\times W$, with $W$ a neighborhood of
the origin in $\NN_x$ (necessarily $G_x$-invariant). 
\end{example}

\begin{definition} \label{versions-linearization} Let $\G$ be a Lie groupoid over $M$, $\OO$ an orbit. 
We say that $\G$ is:
\begin{itemize}
\item[(1)] \textbf{linearizable} at $\mathcal{O}$ if there are neighborhoods of $\OO$, $U\subset M$, $V\subset \NN_{\OO}$, 
and an isomorphism of groupoids $\G|_{U}\cong \mathcal{N}_{\mathcal{O}}(\G)|_{V}$ which is the identity on $\G_{\OO}$. 
\item[(2)] \textbf{semi-invariant linearizable} at $\mathcal{O}$ if $V$ can be chosen invariant.
\item[(3)] \textbf{invariant linearizable} at $\mathcal{O}$ if both $U$ and $V$ can be chosen invariant.
\end{itemize}
(Here, the invariance of $U$ is w.r.t. $\G$, while the one of $V$ is w.r.t. $\mathcal{N}_{\mathcal{O}}(\G)$).
\end{definition}

\begin{remark}\rm \ As we shall see in subsection \ref{Variations}, if $G_x$ is compact ($x\in \OO$), 
one can replace $V= \NN_{\OO}$ in both (2) and (3). If also $\OO$ is compact, the same applies to (1).
\end{remark}

\begin{remark}
It may be tempting to require the conditions above to be realized for arbitrarily small $U$'s. For (1) this is not an issue.
For (2), this requirement would not make a huge difference (e.g., when $G_x$ is compact, (2) implies any way that one can achieve 
$\G|_{U}\cong \NN_{\OO}(\G)$ for arbitrarily small $U$'s, cf. section \ref{Variations}); however, the requirement would
bring in extra-information about the linearization rather than about $\G$ itself.
In the case of (3) however, requiring arbitrarily small $U$'s would be too strong, as it would imply that $\OO$ is stable
(see Remark \ref{remark-s-triviality} above); in particular it would exclude examples such as proper actions of non-compact
Lie groups.
\end{remark}

\begin{example} \rm \ If $\G= G\ltimes M$ comes from a Lie group action (Example \ref{ex-action}), we recover the usual notion of orbit
and isotropy group; the isotropy bundle is the principal $G_x$-bundle $G\rmap G/G_x$; $\mathcal{N}_{\mathcal{O}}(\G)$ is 
the groupoid associated to the action of $G$ on the tube $G\times_{G_x} \mathcal{N}_x$. By the well-known slice theorem  
(or ``tube theorem'', see e.g. (2.4.1) in \cite{DuKo}), $\G$ is invariant
linearizable at $\mathcal{O}$ if the action is proper at $x$.
\end{example}

\begin{example} \rm \ A very instructive example is that of the groupoid $\G= \G(\pi)$ associated to a submersion $\pi: X\rmap Y$
(Example \ref{ex-submersion}). The orbits are precisely the fibers $\mathcal{O}_y= \pi^{-1}(y)$ with $y\in \pi(X) \subset Y$. At such an orbit $\OO_y$, 
\begin{enumerate}
\item $\G$ is linearizable (after all, $\G(\pi)$ is always proper). 
\item $\G$ is semi-invariant linearizable if and only if $\pi$ is semi-trivial around $y$
(i.e., on a neighborhood of $\OO_y$, $\pi$ is equivalent to a projection). 
\item $\G$ is invariant linearizable if and only if $\pi$ is trivial 
around $y$ (i.e., on a saturated neighborhood of $\OO_y$, $\pi$ is equivalent to a projection). 
\end{enumerate}
Recall that, by Ehresmann's theorem, local triviality is achieved when $\pi$ is proper (i.e., when $\G(\pi)$ is $s$-proper, cf. Example \ref{ex-submersion}).
Also, Weinstein shows in \cite{Wein3} that semi-triviality around $y$ is typically achieved when $\mathcal{O}_y$ is of finite type. Actually, this is the only
property of finite type manifolds that will be used in this paper; see also subsection \ref{var-on-hypothesis} (nevertheless, recall \cite{Wein3} that a manifold is called 
of finite type if it admits a proper Morse function with a finite number of critical points). 
\end{example}

\begin{example}\rm \ Another important class of groupoids comes from foliation theory \cite{MoMr}. Let $(M, \mathcal{F})$ be a foliated manifold
and let $\G$ be the associated holonomy groupoid over $M$ (assumed Hausdorff here). Its arrows are holonomy classes of leafwise paths; the orbits are precisely
the leaves of the foliation; the isotropy group at $x$ is the holonomy group $\textrm{Hol}_x$; the isotropy bundle of a leaf $L$ is the holonomy bundle $\tilde{L}\rmap L$;
the resulting linear model at a leaf $L$ is the holonomy groupoid
associated to the linear foliation on $\NN_{L}= \tilde{L}_x\ltimes_{\textrm{Hol}_x} \NN_x$. This is precisely the local model that appears in classical Reeb stability.
The theorem is equivalent to invariant linearization of $\G$; the standard hypothesis are that $L$ is compact and the
holonomy at $x$ is finite; i.e. precisely the $s$-properness of $\G$ at $x$.
\end{example}

\begin{example}\rm \ A similar discussion applies when $\G$ is the symplectic groupoid (see e.g. \cite{DuZu}) associated to a Poisson manifold $(M,\pi)$. The resulting linearizability problems are closely related to normal forms in Poisson geometry. This time
however, all the linearization versions are interesting. For instance, Conn's linearization theorem \cite{Conn2, CrFe} and  
the local form of \cite{Ionut} are related to invariant linearizability of $\G$. 
\end{example}

\subsection{\underline{Existing results}} \label{sec: existing results} 
As we mentioned in the introduction, there has been some confusion regarding the precise statement of the results on the linearization of proper Lie groupoids. By carefully combining the existing proofs in the literature, the best possible statement one arrives at  is
\begin{itemize}
\item If $\G$ is a Lie groupoid over $M$ which is proper and \emph{$s$-locally trivial}\footnote{The appearance of this condition in [20] is somehow enigmatic. On one hand, the theorem
in the fixed point case (Theorem 2.3 in \emph{loc.cit.}) does not state this as an assumption, although it
is used in the proof (see below for details). On the other hand, the general case (i.e. Theorem 2.4 
in \emph{loc.cit.}) is, as the author remarks, an immediate consequence of the Theorem 2.3  
and a result of Weinstein (Theorem 9.1 in \cite{Wein3}). Hence the fact that s-locally triviality appears in the
statement of Theorem 2.4 is rather surprising because, as originally stated, neither Theorem 2.3 of \cite{Zung}
nor Theorem 9.1 of \cite{Wein3} require this condition. However, the condition \emph{should} stay in the theorem
because, as we mentioned, it is used in the proof of Theorem 2.3.}, then $\G$ is linearizable at $\OO$.
\end{itemize}
This is a weaker version of our Corollary \ref{cor1}.

Here is a short summary of the existing results \cite{Wein2, Wein3, Zung, DuZu, Hessel, MoMr, DuKo, CrFe, Ionut, Trent}. 
First, a detailed analysis of \cite{Wein3} reveals that it is entirely about \emph{semi-invariant} linearization (although this is not explicitly stated there). The main parts of \cite{Wein3} are:
\begin{itemize}
\item it reveals the finite type condition in relation with (semi-invariant) linearization. 
\item it shows that, for general proper groupoids, the (semi-invariant) linearization at a finite type orbit can be deduced from the
linearization at fixed points (we call this Weinstein's trick). 
\item it gives many interesting and illuminating examples.
\end{itemize}
The most difficult step was taken by Zung \cite{Zung}, where he states that:
\begin{itemize} 
\item The linearization theorem is true at fixed points of proper Lie groupoids.
\end{itemize}
Note, however, that Zung's proof of the fixed point case starts with several ``simplifying assumption" (the base being a closed invariant ball, and $\G$ being $s$-trivial), for which he refers to \cite{Wein3}. In fact, at the end of page 846 of \cite{Zung}, it is stated that \begin{quote} ``...A simple fact already observed by Weinstein [31] is that, due to the properness, any neighborhood of $x_0$ in $B$ will contain a closed ball-like neighborhood saturated by compact orbits of $\Gamma$. By shrinking $B$ if necessary, we can assume that $B$ is a closed ball, the orbits on $B$ are compact, and the source map $s : \Gamma \to B$ is a trivial fibration." \end{quote}
In this statement, Zung is making use of Theorem 3.3 of \cite{Wein3} which states that a fixed point of a \emph{source locally trivial} proper groupoid is stable. The problem with this is that Zung's version of the theorem (Theorem 2.3 in \cite{Zung}) would itself imply the stability of the fixed point. But the same \cite{Wein3} shows by means of examples that stability can actually fail. This problem is only partially fixed in \cite{DuZu} (see also the final part of our Section \ref{remarks-on-the-proof} bellow).
 
We should also mention \cite{Trent}, where the relevance of Morita equivalence to linearizability was indicated for the first time.  More recently (and independent of our work), \cite{Hessel} also shows that ``Weinstein's trick" can be used to pass from Zung's theorem to the general linearization theorem without using the finite type condition.


Of course, there are also the classical results that correspond to the linearization problem, such as the Reeb stability, for which we used \cite{MoMr},
or the tube theorem, for which we used \cite{DuKo} (in particular, we adopted the weaker notion of ``proper at a point''). The Poisson case \cite{CrFe, Ionut} played a special role (see Remark \ref{remarks-on-the-proof}).

From the literature mentioned above, the most influential work is the one of Weinstein; in some sense, we reinterpret  and rewrite \cite{Wein3} so that it not only applies to arbitrary orbits but it also gives a simple proof to Zung's theorem.  
We would also like to mention that, although most of our work was carried out independently of \cite{Hessel}, we did let ourselves be influenced by it in the final writing.


\subsection{\underline{This paper}} The main points of our paper, which are the subjects of the three main sections of the paper are:
\begin{itemize}
\item A simple, geometric proof of Zung's theorem without the extra ``simplifying assumption" that he makes in his proof.
\item A more conceptual understanding of the linearizability problem; in particular, Weinstein's trick appears as a manifestation
of Morita invariance, and is independent of the extra hypothesis of the orbit being a manifold of finite type. The outcome is the main theorem from the introduction, which is an improvement of the 
theorem that follows from the above mentioned results, and which establishes the linearization result that is valid with only the properness hypothesis on $\G$. 
\item A clarification of the several versions of linearizability (in which invariance conditions are posed), such as those given in Definition \ref{versions-linearization},
and a clarification of their relationship with general linearizability. A brief way to describe our results is that linearizability at $\OO$ gives rise to a new local model
in an {\it invariant} neighborhood. This model is constructed out of an equivariant bundle.  The main conclusion is that the 
passage from general linearizability to its other versions always translates into questions about equivariant bundles (e.g.: the corollaries from the introduction).
\end{itemize}

\subsection*{\underline{Acknowledgements}:} We would like to thank the authors of \cite{Hessel} for providing us with an early copy of their work.
We would also like to thank Ionu\c{t} M\v{a}rcu\c{t} for his comments on the first (written and unwritten) versions of the paper. We extend our thanks also to Matias  L. del Hoyo for pointing out a mistake in a an example in a preliminary version of this paper. We also benefitted from several discussions with David Martinez Torres. Last but not least, we are indebted to Rui Loja Fernandes for several illuminating conversations since the early stages of this project, which lead to a greater understanding of proper groupoids and their linearization.

\section{The fixed point case; a short proof of Zung's theorem}

In this section we give a new proof to Zung's theorem, i.e. the theorem from the introduction in the case
of fixed points $x$ (i.e. when $\OO= \{x\}$):

\begin{theorem} \label{th-Zung}
Let $\G$ be a Lie groupoid over $M$, with fixed point $x\in M$. If $\G$ is proper at $x$, then $\G$ is linearizable at $x$.
\end{theorem}

\subsection{\underline{Restrict to small neighborhoods}} 
We first show the existence of neighborhoods over which $\G$ becomes nicer.



\begin{proposition}\label{prop-fixed-points}
Let $\G$ be a Lie groupoid over $M$ with a fixed point $x\in M$. If $\G$ is proper at $x$
then there exists a neighborhood $W$ of $x$ such that $\G|_{W}$ is proper
and such that there is an open embedding of type
\begin{equation}\label{eq-U} 
(r, s): \G|_{W}\rmap G_x\times W 
\end{equation}
for some $r$ satisfying $r(\gamma)= \gamma$ for $\gamma\in G_x$, $r(1_w)= 1_{G_x}$ for $w\in W$.
\end{proposition}

\begin{proof} Let $\mathcal{N}$ be a tubular neighborhood of $G_x$ in $\G$
and let $r: \mathcal{N}\rmap G_x$ be the projection map. The map
$(r, s): \mathcal{N}\rmap G_x\times M$ is the identity over $G_x$ (where $G_x$ is embedded in the second space as $G_x\times \{x\}$),
and its differential at all points in $G_x$ is an isomorphism. Hence, eventually making $\mathcal{N}$
smaller, we may assume that $(r, s)$ is a diffeomorphism into an open subset of $G_x\times M$
containing $G_{x}\times \{x\}$. Since $G_x$ is compact, any such open set contains one of type $G_x\times B$
with $B$ a neighborhood of $x$ in $M$. Hence, shrinking $\mathcal{N}$ again, we may assume that
we have a diffeomorphism
\[ (r, s): \mathcal{N} \rmap G_x\times s(\mathcal{N}) \]
Next, there is a neighborhood $W$ of $x$ in $M$ such that $\G|_{W}\subset \mathcal{N}$. If not,
then there is a sequence $g_n\in \G$ with $s(g_n), t(g_n)\to x$ and such that $g_n\notin \mathcal{N}$.
Properness at $x$ allows us to assume that $g_n\to g$, with $g\in \G$. Clearly, $g$ must be in $G_x$,
but this is impossible because $\mathcal{N}$ is a neighborhood of $g$ which does not contain any $g_n$.

Choose $W$ such that $\G|_{W}\subset \mathcal{N}$. We remark that $\G|_{W}$ is proper. Indeed, for $K, L\subset W$ compact,
the image of $s^{-1}(K)\cap t^{-1}(L)$ by the diffeomorphism $(r, s)$ is a closed inside $G_x\times K$, hence it is compact. 

Finally, to achieve $r(1_y)= 1$, put $r'(g)= r(g)r(1_{s(g)})^{-1}$.
\end{proof}

\subsection{\underline{Deforming $\G$ into its linearization}}
\label{Deforming G into its linearization}
We will use the notations 
\[ K= G_x, E= K\times \Rr^n\]
and we think of $E$ as a vector bundle over $K$. Hence, for $g= (a, v)\in E$, $\epsilon g= (a, \epsilon v)$.  
By Proposition \ref{prop-fixed-points}, we may assume that $\G$ is proper, $M= \mathbb{R}^n$, $\G$ is an open in $E$ containing the zero section, the source
$s$ is the projection, and the unit at any $v$ is $(1, v)$. 
Next, we deform $\G$ into its linearization $K\ltimes \mathbb{R}^n$. For each $\epsilon$ we consider
\[\G_{\eps} = \set{g \in K \times \Rr^n: \eps g \in \G}\subset K\times \mathbb{R}^n,\]
sitting over $\Rr^n$, with structure maps 
\[s_{\eps}(g)= s(g), t_{\eps}(g) = \frac{1}{\eps} t(\eps g), m_{\eps}(g,h) = \frac{1}{\eps}m(\eps g, \eps h) \text{ and } i_{\eps}(g) =  \frac{1}{\eps} i(\eps g).\]
By definition, $\G_{0}$ is the linearization $K\ltimes \Rr^n$. 
It will be convenient to organize the different $\G_{\eps}$'s into one groupoid. Hence consider
\[\tilde{\G} = \set{(g, \eps) \in E \times \mathbb{R} : \eps g \in \G}\]
and denote its source, target, multiplication and inversion by $\sigma$,  $\tau$, $\mu$ and $\iota$, respectively. Hence
$\sigma(g, \eps) = (s_{\eps}(g), \eps)$, $\mu((g, \eps),(h,\eps))=(m_{\eps}(g,h),\eps)$ etc. 

\begin{lemma} $\tilde{\G}$ is open in $E \times \Rr$ and it is a proper Lie groupoid. 
\end{lemma}

\begin{proof} 
The openness of $\tilde{\G}$ follows from the one of $\G$ and continuity of $(\eps, g)\mapsto \eps g$.
For the smoothness of the structure maps, the only problem is continuity at $\eps= 0$. For the target we check that, for
$g_{\eps}= (a_{\eps}, v_{\eps})$ depending smoothly on $\eps$ near $0$, 
$t_{\eps}(g_{\eps})$ goes to $t_0(g_0)$ when $\eps\to 0$. We have $t_0(g_0)= a_0v_0$ and we use the description of the isotropy action in terms of curves
(see the introduction). 
We choose the curve $g(\eps)= (a_{\eps}, \eps v_{\eps})$, which has $g(0)= (a_{0},0)$ and the derivative of $s(g_{\eps})$ at $0$ is $v_0$.
Hence $a_0v_0$ is the derivative at $0$ of $t(g_{\eps})$, i.e. the limit of $\frac{1}{\eps}t(g_{\eps})= t_{\eps}(g_{\eps})$. 
For the multiplication, we choose $g_{\eps}$ as above and $h_{\eps}= (b_{\eps}, w_{\eps})$ s.t. they are composable in $\G_{\eps}$. 
The multiplication of $\G$ is of type $m(g, h)= (m'(g, h), s(h))\in K\times \Rr^n$; then,  
\[ m_{\eps}(g_{\eps}, h_{\eps})= (m'(\eps g_{\eps}, \eps h_{\eps}), s(h_{\eps}))\to (m'((a_0, 0)m (b_0, 0)), v_0)= (a_0b_0, w_0),\]
when $\eps\to 0$. This is precisely which is $m_{0}(g_0, h_0)$.

For properness, let $(g_n, \eps_n)$ be a sequence in $\tilde{\G}$ with convergent source and target. Writing $g_n= (a_n, v_n)$, that means that,
for some $v, w\in \Rr^n$, $\eps\in \Rr$,  
\[ \sigma(g_n, \eps_n)= (v_n,\eps_n)\to (v, \eps),  \tau(g_n, \eps_n)= (\frac{1}{\eps_n}t(\eps_n g_n),\eps_n)\to (w, \eps).\]
To show that $(g_n, \eps_n)$ has a subsequence  convergent in $\tilde{\G}$, we may assume that
$g_n= (a_n, v_n)\to (a, v)=: g$ in $E$, for some $a\in K$ (since $K$ is compact). To show that $(g, \eps)\in \tilde{\G}$, 
note that both the source as well as the target of $\eps_ng_n\in \G$ are convergent (to $\eps u$ and $\eps w$),
hence, since $\G$ is proper, the limit $\eps g$ must be in $\G$.
\end{proof}

\subsection{\underline{Multiplicative vector fields}} To relate $\G_{\eps}$ at different $\eps$'s, we will use flows of multiplicative vector fields. Recall that, in general, 
given a Lie groupoid $\HH$ over $N$, a vector field $X$ on $\HH$ is called \textbf{multiplicative} if, as a map $X: \HH\rmap T\HH$,
it is a morphism of groupoids i.e., it is $s$ and $t$-projectable to a vector field on $N$ and 
\[ (dm)_{(g, h)}(X_g, X_h)= X_{gh} \]
for all $g, h\in \HH$ composable. We will need the following basic property \cite{MkXu}.

\begin{lemma}\label{mult-flow}  If $X$ is a multiplicative vector field on the groupoid $\HH$ then its flow preserves the groupoid structure. More precisely,
for each $\epsilon$, the space $\mathcal{D}_{\epsilon}(X)\subset \HH$ of points at which the flow $\phi_{X}^{\epsilon}$ is defined is an open subgroupoid of $\HH$, 
with base the similar space $\mathcal{D}_{\epsilon}(V)\subset M$ of the base vector field $V$, and
\[ \phi_{X}^{\eps}: \mathcal{D}_{\eps}(X)\rmap \HH\]
is a morphism of groupoids covering the flow of $V$.
\end{lemma}

Next, for the groupoids $\G_{\epsilon}$, there is a small unpleasant problem: since $t_{\epsilon}$ varies with $\epsilon$, the space
of composable arrows varies as well. Since $s_{\epsilon}= s$ stays fixed, it is better to consider, for any groupoid $\HH$, 
\[\HH^{[2]} =\set{(p,q) \in \HH\times \HH: s(p) = s(q)},\]
\[\bar{m} : \HH^{[2]} \to \HH, \quad \bar{m}(p,q) = m(p, i(q))= pq^{-1}.\]
The key remark is that the source $s$, the unit map $u$ and the modified multiplication $\bar{m}$ encode the entire groupoid structure. 
Indeed, 
\[ t(g)  = u^{-1}\bar{m}(g,g),\ \  i(g) = \bar{m}(u\circ s(g), g), \ \  m(g,h) = \bar{m}(g, h^{-1}).\]
Of course, we can reformulate the notion of groupoid in terms of $(\HH, s, u, \bar{m})$.
 We only need  here the following immediate consequence of the last set of formulas.

\begin{lemma} Given the Lie groupoids $\HH$ over $M$, $\KK$ over $N$, and given two smooth maps
 $F: \HH\rmap \K$, $f: M\rmap N$ which are compatible with $s$ and $u$ (i.e. $sF= fs$, $Fu= uf$), 
$F$ is a groupoid morphism if and only if
\[F(\bar{m}(p,q)) = \bar{m}(F(p), F(q)) \text{ for all } p,q \in \HH^{[2]}.\]
\end{lemma}

In particular, we will use this lemma to check multiplicativity of vector fields.

\subsection{\underline{The deformation cocycle}} Let us now return to our $\tilde{\G}$. Note that 
\[ \G^{[2]}_{\eps}\subset F= K\times K\times \Rr^n\]
and they fit together into the open subspace $\tilde{\G}^{[2]}$ of $F\times \Rr$.  The idea is
find a multiplicative vector field on $\tilde{\G}$ with second component $\partial_{\eps}$, so that its flow moves the
parameter by translation. The obvious one, $\partial_{\eps}$, is not multiplicative. 
As it will become clear in a moment,
the defect of its multiplicativity is measured by
\begin{equation}\label{eq: xi-lambda}
\G^{[2]}_{\lambda}\ni (p, q)\mapsto \xi_{\lambda}(p,q): =  \frac{\d}{\d \eps}|_{\eps = \lambda}\bar{m}_{\eps}(p,q)\in T_{\bar{m}_{\lambda}(p,q)}\G_{\lambda}.
\end{equation}
Note that, since $\tilde{\G}^{[2]}$ is open, these derivatives are well defined. Now, a vector field $\tilde{X}$ on $\tilde{\G}$ with second component $\partial_{\eps}$
can be written as
\begin{equation}
\label{the-X-lambda} 
\tilde{X}_{p, \lambda}= X_{p}^{\lambda}+ \partial_{\lambda},
\end{equation}
where each $X^{\lambda}$ is a vector field on $\G_{\lambda}$.

\begin{lemma}\label{last-cond-mult} The vector field (\ref{the-X-lambda}) is a multiplicative vector field if and only if,
for each $\lambda$, there is a vector field $V^{\lambda}$ on the base $\Rr^n$ such that 
\begin{equation}\label{eq-X-0} 
(ds)_{g}(X^{\lambda}_{g})= V_{s(g)}^{\lambda}, \ (du)_x(V^{\lambda}_{x})= X^{\lambda}_{u(x)}
\end{equation}
for all $g\in \G_{\lambda}$, $x\in \Rr^n$ and, 
for all $(p, q)\in \G^{[2]}_{\lambda}$,
\begin{equation}\label{eq-X} 
(d\bar{m}_{\lambda})_{p, q}(X_{p}^{\lambda}, X_{q}^{\lambda})= X_{\bar{m}_{\lambda}(p, q)}^{\lambda} - \xi_{\lambda}(p, q).
\end{equation}
\end{lemma}

\begin{proof} We use the last lemma. Fix $p$, $q$ and $\lambda$ and write $\tilde{X}_{p,\lambda}= \dot{u}(\lambda)$, $\tilde{X}_{q,\lambda}= \dot{v}(\lambda)$
($u(\lambda)= p$, $v(\lambda)= q$). The failure to the multiplicativity of $\tilde{X}$ at $p$, $q$ and $\lambda$ is
\begin{align*}
&  (d\bar{\mu})(\tilde{X}_{p, \lambda}, X_{q,\lambda}))- \tilde{X}_{\bar{\mu}((p, \lambda), (q, \lambda))} = \\
= & \frac{\d}{\d \eps}|_{\eps = \lambda} \bar{\mu}((u(\eps), \eps), (v(\eps), \eps))- (X_{\bar{m}_{\lambda}(p, q)}^{\lambda}, \partial_{\lambda})\\
= &   \frac{\d}{\d \eps}|_{\eps = \lambda} (\bar{m}_{\epsilon}(u(\eps), v(\eps)), \eps) -  (X_{\bar{m}_{\lambda}(p, q)}^{\lambda}, \partial_{\lambda})\\
= &  (\frac{\d}{\d \eps}|_{\eps = \lambda} \bar{m}_{\epsilon}(u(\eps), v(\eps))- X_{\bar{m}_{\lambda}(p, q)}^{\lambda}, 0) \\
= & ((\xi_{\lambda}(p, q)+ (d\bar{m}_{\lambda})_{p, q}(X_{p}^{\lambda}, X_{q}^{\lambda})- X_{\bar{m}_{\lambda}(p, q)}^{\lambda}, 0).
\end{align*}
\end{proof}


\begin{lemma} For any $u, v, k\in \G_{\lambda}$ such that $(u, k), (v, k)\in \G^{[2]}_{\lambda}$,
\begin{equation}\label{eq-xi} 
(d\bar{m}_{\lambda})(\xi_{\lambda}(u, k), \xi_{\lambda}(v, k))= \xi_{\lambda}(u, v)- \xi_{\lambda}(\bar{m}_{\lambda}(u, k), \bar{m}_{\lambda}(v, k)).
\end{equation}
\end{lemma}

\begin{proof} Fix $\lambda$, $u, v, k$. Due to the openness again, for $\epsilon$ close to $\lambda$, $(u, k), (v, k)\in \G^{[2]}_{\epsilon}$ and the associativity of
$m_{\eps}$ gives us
\[\bar{m}_{\eps}(\bar{m}_{\eps}(u, k), \bar{m}_{\eps}(v, k))= \bar{m}_{\eps}(u, v).\]
Taking the derivative with respect to $\epsilon$ at $\lambda$, we obtain the desired formula.
\end{proof}

\subsection{\underline{Averaging}} 
Recall that any proper Lie groupoid $\HH$ over a manifold $N$ admits a family 
of linear maps (constructed from a Haar system and a cut-off function \cite{Crainic})
\[ \int_{t^{-1}(x)} : C^{\infty}(t^{-1}(x)) \rmap \mathbb{R} , f\mapsto \int_{t^{-1}(x)} f\stackrel{\textrm{notation}}{=} \int_{x} f(g) dg\]
parametrized smoothly by $x\in N$ (i.e., applying them to a smooth function on $\HH$ produces a smooth function on $N$), normalized
(sends the constant function $1$ to $1$) and left invariant (for $a: x\rmap y$ in $\HH$, the composition with the left translation $L_a: t^{-1}(x)\rmap t^{-1}(y)$ does not change the integral). We apply this to $\tilde{\G}$; hence the resulting integral brakes into integrals $\int_{x}^{\lambda}$ performed on $t_{\lambda}^{-1}(x)$. 
We define
\begin{equation}
\label{X-integral} 
X^{\lambda}_{g}= \int_{s(g)}^{\lambda} \xi_{\lambda}(m_{\lambda}(g, k), k) dk \in T_g\G_{\lambda}.
\end{equation}

\begin{lemma} 
The resulting vector field $\tilde{X} \in \X(\tilde{G})$ is multiplicative.
\end{lemma}


\begin{proof} We use Lemma \ref{last-cond-mult}. Fix $\lambda$. As vector field $V^{\lambda}$ on the base we take:
\[ V^{\lambda}_{x} = \int_{x}^{\lambda}v^{\lambda}(k)dk \in T_x\Rr^n, \textrm{where}\    v^{\lambda}(k)= \frac{\d}{\d \eps}|_{\eps = \lambda}t_{\eps}(k) .\]
To check (\ref{eq-X-0}), using the linearity of the integral, it suffices to show:
\[ (ds)(\xi_{\lambda}(p, q))= v^{\lambda}(q), \  \xi_{\lambda}(q, q)= (du)(v^{\lambda}(q)) .\]
But this is immediate from the definition of $\xi_{\lambda}$ and the fact that $s(m_{\eps}(p, q))= t_{\eps}(q)$, $m_{\eps}(q, q)= u(t_{\eps}(q))$
(... and the independence of $s$ and $u$ of $\eps$!).
To check (\ref{eq-X}), since $\lambda$ is fixed, to keep the formulas simpler, let's just write $gh$ for the multiplication $m_{\lambda}(g, h)$ and $g^{-1}$ for the inverse $i_{\lambda}(g)$. 
Computing $(\d \bar{m}_{\lambda})(X_{p}^{\lambda}, X_q^{\lambda})$, we obtain
\begin{align*}
&  (d\bar{m}_{\lambda})(\int_{s(q)}^{\lambda} \xi_{\lambda}(pk, k)dk , \int_{s(q)}^{\lambda}\xi_{\lambda}(qk, k)dk) =\int_{s(q)}^{\lambda} (d\bar{m}_{\lambda})(\xi_{\lambda}(pk, k), \xi_{\lambda}(qk, k)) dk  \\
& \stackrel{(\ref{eq-xi})}{=} \int_{s(q)}^{\lambda} (\xi_{\lambda}(pk, qk)- \xi_{\lambda}(p, q)) dk= \int_{s(q)}^{\lambda} \xi_{\lambda}(pk, qk) dk - \xi_{\lambda}(p, q) \\
& \stackrel{k'= qk}{=} \int_{t(q)}^{\lambda} \xi_{\lambda}(pq^{-1}k', k') dk- \xi_{\lambda}(p, q) = X_{\bar{m}_{\lambda}(p, q)}- \xi_{\lambda}(p, q).
\end{align*}
\end{proof}

\subsection{\underline{End of the proof of  Theorem \ref{th-Zung}}} We use Lemma \ref{mult-flow} for $\tilde{X}$. Let $\phi^{\eps}$ be its flow,
$\psi^{\eps}$ the flow on the base.
 Note that $\phi^{\eps}(g, 0)= (g, \eps)$ for $g\in K\times \{0\}$. Hence we find a neighborhood $B$ of the origin such that
$\phi^{1}(g, 0)$ is defined for all $g\in K\times B$; $\phi^{1}$ defines a diffeomorphism from $K\times B\times \{0\}$ to an open set in $\G_1$ containing
$K\times \{0\}\times \{1\}$. Choose $U_1\subset \mathbb{R}^n$ containing the origin such that $K\times U_{1}\times \{1\}\subset \phi^{1}(K\times B\times \{0\})$. In particular,
$\G_{1}|_{U_1}\subset \phi^{1}(K\times B\times \{0\})$. Let $U_0\subset B$ such that $\psi^{1}$ maps $U_{0}\times \{0\}$ onto $U_1\times \{1\}$. 
Since $\phi^{1}$ covers $\psi^{1}$ with respect to the source and the target, it is clear that $\phi^1$ sends $\G_{0}|_{U_0}\subset K\times B\times \{0\}$ (diffeomorphically) onto $\G_{1}|_{U_1}$. 

\subsection{\underline{Some remarks on the proof}}
\label{remarks-on-the-proof}
The idea of the proof is rather standard and can already be found in the work of Palais on rigidity of Lie group actions \cite{Palais}. 
In the context of groupoids, it was used by Weinstein in the proof of his Theorem 7.1 from \cite{Wein3}. More loosely, our proof can be viewed as a global version of the Moser deformation argument from  \cite{CrFe, Ionut}. Actually, we 
realized that such a proof may be possible precisely when we found out that the cohomological obstruction arising from the Moser deformation argument is, in the integrable case, the shadow of a cocycle on the groupoid (i.e. the proof of Corollary 2.1, in paragraph 5.1 of \cite{Ionut}). 

Note also that our proof hides a cohomology theory that underlies deformations
of proper Lie groupoids, which makes the relationship with Palais' work explicit and which confirms 
some of Weinstein's expectations (Remark 7.3 in \cite{Wein3}). Details in this direction will be given elsewhere \cite{Nuno}.

Finally, note that our proof is very different from Zung's \cite{Zung}. Even at the start, we only assumed that $\G$ is open in $K\times \Rr^n$; the
stronger assumption that the two spaces are the same, made at the beginning of  subsection 2.2 in \cite{Zung}, would have simplified the argument,
but the only way we can support such an assumption is by first proving the theorem.


\section{General remarks; the end of the proof of Theorem \ref{main-theorem}}

This section contains several general remarks on the nature of the linearizability problem 
(Morita invariance) and on properness and its relevance to the existence of slices. These, combined
with our version of Zung's theorem, immediately imply Theorem \ref{main-theorem}.

\subsection{\underline{On Morita invariance}}
\label{subsection-Morita}
In this subsection we explain the relevance of the notion of Morita equivalence to linearization.
Let $\G$ be a Lie groupoid over $M$ and $\HH$ a Lie groupoid over $N$. Recall \cite{MoMr, BW, Crainic} that a \textbf{Morita equivalence} between $\G$ and $\HH$ 
is given by a principal $\HH$-$\G$ bi-bundle, i.e. a manifold $P$ endowed with  
\begin{enumerate}
\item Surjective submersions $\alpha: P\rmap N$, $\beta: P\rmap M$,
\item A left action of $\HH$ on $P$ along the fibers of $\alpha$, which makes $\beta: P\rmap M$ into a principal $\HH$-bundle
in the sense that the action map $\HH\times_{N}P\rmap P\times_{M}P$, $(h, p)\mapsto (hp, p)$, is a diffeomorphism,
\item Similarly, a right action of $\G$ on $P$ along the fibers of $\beta$, which makes $\alpha$ into a principal $\G$-bundle,
\end{enumerate}
and the left and right actions are required to commute. One says that $\G$ and $\HH$ are Morita equivalent if such a bi-bundle exists.

\begin{example}\rm \ Here are some standard examples of Morita equivalences.
Two Lie groups, viewed as Lie groupoids over the one-point space, are Morita equivalent if and only if they are isomorphic. A transitive Lie groupoid
(i.e. with the property that it has only one orbit) is Morita equivalent to the isotropy Lie group of a(ny) point in the base; the bi-bundle is the fiber of $s$ above the point.
The groupoid $\G(\pi)$ associated to a submersion $\pi: M\rmap N$ (Example \ref{ex-submersion}) is Morita equivalent
to $\pi(M) \subset N$, viewed as a groupoid with identity arrows only; the bi-bundle is $P= M$. The action groupoid associated to an action of a Lie group
$G$ on $M$ (Example \ref{ex-action}), for proper free actions, is Morita equivalent to $N= M/G$; the bi-bundle is $P = M$. 
Intuitively, two Lie groupoids are Morita equivalent if they have the same ``transversal geometry''. 
\end{example}

\begin{example}\label{example-saturation}\rm \ If $U$ is an open set in $M$, then the restriction of $\G$ to the saturation of $U$ (an invariant open set) is Morita equivalent to $\G|_{U}$;
a bi-bundle is $t^{-1}(U)$ -- the space of arrows of $\G$ ending at points of $U$, endowed with $\alpha= s$, $\beta= t$ and the natural actions
coming from the multiplication of $\G$. Hence one may say that up to Morita equivalences, when restricting a groupoid $\G$ to an open set $U$, arbitrary open sets are as good as
invariant ones. It follows that if $\G$ linearizable at $\OO$ the its restriction to an invariant neighborhood of $\OO$ is Morita equivalent to the restriction of the local model to an invariant neighborhood of $\OO$. In other words, linearizability of $\G$ at $\OO$ is ``invariant up to Morita equivalence''.
\end{example}

\begin{example}\label{example-M-isotropy}\rm \ For the main theorem, the most important example of Morita equivalence arises in the presence of
transversals (see (2) of Lemma \ref{lemma-transversals}). Here are a few more occurrences. Fix $\G$ over $M$ and let $\OO$ be the orbit through $x$. 
Then, since $\G_{\OO}$ is transitive, it is Morita equivalent to $G_x$. Similarly, the linearization $\NN_{\OO}(\G)$ is Morita equivalent to the action groupoid $G_x\ltimes \NN_x$  ((Example \ref{ex-action}))
associated to the isotropy action; a bi-bundle is $P_x\times \NN_x$, with $\alpha$ the second projection, $\beta$ the quotient map and the actions are the canonical ones. 

Finally, we will also use the following basic property of Morita equivalences: if $P$ realizes a Morita equivalence between $\G$ and $\HH$, 
then $\G$ can be re-constructed from $\HH$ and $P$. Relevant here is 
just the free and proper left action of $\HH$ on $P$, along the submersion $\alpha: P\rmap N$.
The resulting ``gauge groupoid'' is the quotient of the groupoid $\G(\alpha)$ associated to $\alpha$ (Example \ref{ex-submersion}), modulo the diagonal action of $\HH$:
\[
\xymatrix{
\textrm{Gauge}(P):= \G(\alpha)/\HH= (P\times_{\alpha} P)/\HH \ar[r] \ar@<+1ex>[r]&   P/\HH,  }
\]
When $P$ induces an equivalence between $\G$ and $\HH$, then the map $P\times_{\alpha} P\rmap \G$ which associates to a pair $(p, q)$ the unique $g\in \G$ such that $pg= q$ induces an isomorphism from the gauge groupoid to $\G$.  
Of course, this is just a generalization of Example \ref{ex-principal-bundles} from Lie groups to Lie groupoids. 
\end{example}

Any bi-bundle $P$ as above induces a 1-1 correspondence between the orbits of $\G$ and those of $\HH$.
Explicitly, two orbits $\OO\subset M$, $\OO'\subset N$ correspond to each other if and only if 
$\beta^{-1}(\OO)= \alpha^{-1}(\OO')$; in this case we say that \textbf{$\mathcal{O}$ and $\mathcal{O}'$ are $P$-related}. 

\begin{remark}\label{moving-opens}\rm \ One obtains a homeomorphism between the spaces of orbits of $\G$ and $\HH$. 
Let us identify the lattice of open sets in the orbit spaces with the lattice of invariant open sets in the
bases. The remark is that $P$ induces a 1-1 correspondence between $\HH$-invariant open sets $V\subset N$, $\G$-invariant open sets $U\subset M$ and bi-invariant open sets $D\subset P$. Explicitely, $D= \alpha^{-1}(V)= \beta^{-1}(U)$. Note also that $D$ defines a Morita equivalence between
$\G|_{U}$ and $\HH|_{V}$.
\end{remark}

It is well-known that properness is a Morita invariant notion (cf. e.g. \cite{MoMr});
here we show its pointwise version. Say that \textbf{$x$ and $y$ are $P$-related} if $\OO_x$ and $\OO_y$ are.

\begin{proposition}\label{prop: morita invariance of properness} Let $\G$ be a Lie groupoid over $M$, $\HH$ a Lie groupoid over $N$, $P$ a Morita equivalence between them. 
If $x\in M$, $y\in N$ are $P$-related, then $\G$ is proper at $x$ if and only if $\HH$ is proper at $y$.
\end{proposition}

\begin{proof} 
Assume that $\HH$ is proper at $y$. Let $(g_n)_{n\geq 1}$ be a sequence in $\G$ with $s(g_n), t(g_n)\to x$. We will show that it admits a convergent subsequence.
Let $p\in P$ with $\beta(p)= x$, $\alpha(p)= y$. 
Since $\beta(p)= x$ and $\beta$ is a submersion (hence it has local sections), we find $p_n, q_n\in P$ converging to $p$ such that $\beta(p_n)= t(g_n)$,
$\beta(q_n)= s(g_n)$. Since $\beta(p_ng_n)= \beta(q_n)$ (they are both $s(g_n)$), we find $h_n\in H$ such that $p_ng_n= h_nq_n$. In turn, this implies that $s(h_n)= \alpha(q_n) \to \alpha(p)= y$ and $t(h_n)= \alpha(h_nq_n)= \alpha(p_ng_n)= \alpha(p_n)\to \alpha(p)= y$. Hence we may assume that $h_n$ converges to
some $h\in \HH$. But then, by the diffeomorphism $P\times_{M}\G\cong P\times_{N}P$, $(p_n, g_n)$ is sent to $(p_n, p_ng_n)= (p_n, h_nq_n)$ which converges to
$(p, hq)$; hence $(p_n, g_n)_{n\geq 1}$ is convergent, hence $(g_n)_{n\geq 1}$ is. 
 \end{proof}

\begin{remark} \rm \ Given $P$, $x$ and $y$ are $P$-related if and only if there exists $p\in P$ such that
$\beta(p)= x, \alpha(p)= y$. Any such $p$ induces an isomorphism between the isotropy groups,
$\phi_p: G_x\rmap H_y$, uniquely determined by $pg= \phi(g)p$ for all $g\in G_x$. Similarly, $p$ induces a
$\phi_p$-equivariant isomorphism between the isotropy representations, $\Phi_p: \mathcal{N}_x\rmap \mathcal{N}_y$, 
uniquely determined by $\Phi_p([(d\beta)_p(X_p)])= [(d\alpha)_p(X_p)]$ for all $X_p\in T_pP$.
Hence the resulting action groupoids are isomorphic; in particular, the resulting linearizations
$\NN_{\OO_x}(\G)$ and $\NN_{\OO_y}(\HH)$ are Morita equivalent.
\end{remark}

Finally, we show that also the notion of linearizability is Morita invariant. This gives a generalization, and a more conceptual interpretation 
to Theorem 9.1 of \cite{Wein3} and of Theorem Theorem 4.2 of \cite{Hessel}
(note however that we do not assume any topological condition on the orbits, nor properness, and not even the compactness of the isotropy groups, and this requires
extra-care in the proof).

\begin{proposition}[Weinstein's Trick]\label{Morita-equivalence} Let $P$ be a Morita equivalence between $\G$ and $\HH$, and assume that $\OO$ and $\OO'$ are $P$-related orbits.
Then $\G$ is linearizable at $\mathcal{O}$ if and only if $\HH$ is linearizable at $\mathcal{O}'$.  
\end{proposition}

\begin{proof}
Assume that $\HH$ is linearizable at $\mathcal{O}'$; write $\mathcal{O}= \mathcal{O}_x$, $H= G_x$, $\NN= \NN_x$.  
We claim that, after passing to a (invariant) neighborhood
of $\OO$ in $M$,  we may assume
\begin{equation}\label{reduction} 
\HH= H\ltimes W, \text{ and } y= 0 \  \textrm{with} \ W\subset \mathcal{N}\ \ \textrm{an}\ H-\textrm{invariant\ neighborhood\ of} \ 0 .
\end{equation}
Note that this is a purely Morita equivalence statement; it follows immediately
using transitivity of Morita equivalences, allways passing to invariant open sets (cf. Example \ref{example-saturation}) and 
moving in between the restrictions of our groupoids to invariant open sets in $N$ and $M$ (cf. Remark \ref{moving-opens}). 
Next, we use the gauge construction (Example \ref{example-M-isotropy}) 
to reconstruct $\G$ from $P$ and (\ref{reduction}). To emphasize the relevant structure, 
we consider the category $\mathcal{C}= \mathcal{C}_{H}(\NN)$ of manifolds $Q$ endowed with a 
free and proper left action of $H$ and an equivariant submersion $\alpha: Q\rmap \NN$ with image containing the origin (see also subsection \ref{The new local model}).
Hence $\textrm{Gauge}(Q)= \G(\alpha)/H$ (over $Q/H$) is defined for any object of $\mathcal{C}$, and we may assume that 
$\G= \textrm{Gauge}(P)$ and think about $M$ as $P/H$ and of $\OO$ as $P_0/H$, where $P_0= \alpha^{-1}(0)$. 
It is clear that any open embedding in $\mathcal{C}$, $Q\hookrightarrow P$, induces an open set
$U= Q/H\subset M$ and $\textrm{Gauge}(Q)= \textrm{Gauge}(P)|_{U}= \G|_{U}$. Moreover, we want
$P_0\subset Q$ so that $U$ contains $\OO$. We claim that
it suffices to find such an open set $Q\subset P$ with the property that $(Q, \alpha)$ also embeds openly in the product $(P_0\times \NN, \textrm{pr}_2)$ by some embedding $\phi$. Indeed, $\phi$ would induce an open set $V\subset P_0\times_{H} \NN= \NN_{\OO}(\G)$ 
and an isomorphism between $\textrm{Gauge}(Q)$ and $\textrm{Gauge}(P_0\times \NN)|_{V}$; i.e. an isomorphism between $\G|_{U}$ and
$\NN_{\OO}(\G)|_{V}$. 

Hence we are left with finding an $H$-equivariant open neighborhood $Q$ of $P_0$ in $P$ together with an equivariant open embedding
of type $\phi= (r, \alpha): Q\rmap P_0\times \beta(Q)$. 
To construct $Q$, we start with an $H$-equivariant tubular neighborhood $r: \mathcal{U}\rmap P_0$ of $P_0$ in $P$ (which exists since the action is free and proper). Then the map
\[ \phi= (r, \alpha): \mathcal{U}\rmap P_0\times \beta(\mathcal{U}) \]
is equivariant, is the identity on $P_0$, and is a local diffeomorphism at the points in $P_0$. The existence of $Q$ follows from the 
following folklore result.
\end{proof}

\begin{lemma} If a Lie group $H$ acts freely and properly on manifolds $\mathcal{U}$ and $\mathcal{V}$, $A$ is an invariant submanifold of both $\mathcal{U}$ and $\mathcal{V}$, and
$\phi: \mathcal{U}\rmap \mathcal{V}$ is a smooth map such that $\phi|_{A}= \textrm{Id}_A$ and $\phi$ is a local diffeomorphism at the points of $A$,
then there is an $H$-invariant neighborhood $Q\subset \mathcal{U}$ of $A$ 
such that $\phi: Q\rmap \phi(Q)$ is a diffeomorphism.
\end{lemma}

\begin{proof} The set of points at which the differential of $\phi$ is an isomorphism is open and invariant, hence we may assume that $\phi$ is a local diffeomorphism on the entire $\mathcal{U}$. 
The theorem is well-known if $H$ is trivial (cf. e.g. \cite{Lang}). We apply this case to the induced map 
$\overline{\phi}: \mathcal{U}/H\rmap \mathcal{V}/H$ and $A/H$. Let $\overline{Q}$ be the resulting open set in $\mathcal{U}/H$ and
put $Q\subset \mathcal{U}$ its preimage by the quotient map. The freeness of the action implies that $\phi|_{Q}$ is injective, hence also diffeomorphism onto its image.
\end{proof}

\begin{corollary}\label{cor-comp-lin} Assume that $G_x$ is compact. Then $\G$ is linearizable at $\OO$ if and only if its restriction to an some open neighborhood of $\OO$ is Morita equivalent to the action groupoid $G_x\ltimes \NN_x$ associated to the isotropy action, in such a way that $\OO$ is related to $0\in \NN_x$. 
\end{corollary}


\subsection{\underline{On properness and slices}}
Next, we discuss properness and slices at arbitrary points. Given a Lie groupoid $\G$ over $M$ and $x\in M$, we will call a \textbf{slice at $x$} any embedded submanifold
$\Sigma\subset M$ containing $x$ with the following properties:
\begin{enumerate}
\item $\Sigma$ is transversal to every orbit that it meets.
\item $\Sigma$ is of dimension complementary to the dimension of $\mathcal{O}_x$ and $\Sigma\cap \mathcal{O}_x= \{x\}$. 
\end{enumerate}

The main use of slices comes from the following:

\begin{lemma}\label{lemma-transversals} Let $\G$ be a Lie groupoid over $M$, $\Sigma$ a slice at $x$. Then
\begin{enumerate}
\item[(1)]  $\G|_{\Sigma}$ is a Lie groupoid over $\Sigma$ which has $x$ as fixed point.
\item[(2)] The $\G$-saturation $U$ of $\Sigma$ is open and $\G|_{\Sigma}$ is Morita equivalent to $\G|_{U}$.
\item[(3)] $\G$ is proper at $x$ if and only of $\G|_{\Sigma}$ is proper at $x$.
\end{enumerate}
\end{lemma}

\begin{proof} Since $t$ is a submersion, $t^{-1}(\Sigma)$ is a submanifold of $\G$. We claim that,
\[\phi:= s|_{t^{-1}(\Sigma)}: t^{-1}(\Sigma) \rmap M,\] 
is a submersion. This will imply that $\G|_{\Sigma}= \phi^{-1}(\Sigma)$ is a submanifold of $\G$, hence a Lie groupoid;
also, the image of $\phi$, which is precisely the saturation $U$ of $\Sigma$, will be open in $M$. To show that $\phi$ is a submersion,
we compute the dimension of the image of $(d\phi)_g$, for $g\in t^{-1}(\Sigma)$. Denoting $a= s(g)$, $b= t(g)$, the dimension to 
compute is 
\[ \textrm{dim}(T_gt^{-1}(\Sigma))- \textrm{dim}(\textrm{Ker}(d\phi)_g) .\]
For the first term, we use that the codimension of $t^{-1}(\Sigma)$ in $\G$ equals
the codimension of $\Sigma$ in $M$, i.e. the dimension of $\mathcal{O}_x$. For the second term, we use that
\[ \textrm{Ker}(d\phi)_g= \{X\in T_g(s^{-1}(a)): (dt)_g(X)\in T_b\Sigma \}.\]
Let us denote by $A$ the restriction to $M$, via the unit map, of the vector bundle over $\G$ of vector tangent
to the $s$-fibers, and by $\rho: A\rmap TM$ the bundle map given by the differential of the target ($A$ is just the Lie algebroid of $\G$ and $\rho$ is its anchor). Using right translations,
$T_g(s^{-1}(a))$ is isomorphic to $A_b$ and we obtain
\[ \textrm{Ker}(d\phi)_g\cong \{V\in A_b: \rho(b)\in T_b\Sigma \}.\]
Hence 
\[ \textrm{dim}(A_b)- \textrm{dim}(\textrm{Ker}(d\phi)_g)= \textrm{dim}(T_b\mathcal{O}_b)- \textrm{dim}(T_b\mathcal{O}_b\cap T_b\Sigma)\]
which, since $\Sigma$ is transversal to $\mathcal{O}_b$, equals to $\textrm{dim}(T_bM)- \textrm{dim}(T_b\Sigma)= \textrm{dim}(\mathcal{O}_x)$.
All together, the dimension of the image of $(d\phi)_g$ equals to
\[ (\textrm{dim}(\G)- \textrm{dim}(\mathcal{O}_x))- (\textrm{dim}(A_b)- \textrm{dim}(\mathcal{O}_x))= \textrm{dim}(M),\]
proving that $\phi$ is a submersion. To see that $\G|_{U}$ is Morita equivalent to $\G|_{\Sigma}$, we consider the bi-bundle
$P= t^{-1}(\Sigma)$, with $\alpha$ given by the target and $\beta$ given by the source map (a submersion by the discussion above!),
with the left action of $\G|_{\Sigma}$ and the right one of $\G|_{U}$ given by multiplication of $\G$.
 Finally, (3) follows from Morita invariance of properness. 
\end{proof}

\begin{proposition}\label{exists-slice} Assume that the Lie groupoid $\G$ over $M$ is proper at $x$. Then
\begin{enumerate}
\item[(1)] $\mathcal{O}_x$ is an embedded submanifold of $M$.
\item[(2)] There exists a slice $\Sigma$ at $x$.
\end{enumerate} 
\end{proposition}

\begin{proof} For the first part, we show that any sequence $(a_n)_{n\geq 1}$ of points in $\mathcal{O}_x$ which converges in $M$ to $a\in \mathcal{O}_x$
has a subsequence which is convergent to $a$ in the topology of $\mathcal{O}_x$, i.e. in the topology induced by the submersion $t: s^{-1}(x) \rmap s^{-1}(x)/G_x\cong \mathcal{O}_x$.
Since $a$ and the $a_n$'s are in the orbit of $x$, we find arrows $g_n: x\rmap a_n$ and $g: a\rmap x$ of $\G$. Since $s: \G\rmap M$ is a submersion, we
find a local section $\sigma: M\rmap \G$ such that $\sigma(a)= g$. Since $a_n\to a$, we may assume that $a_n$ is in the domain of $\sigma$. But then $u_n:= \sigma(a_n)g_n$
have $s(u_n)= s(g_n)= x$ and $t(u_n)= t(\sigma(a_n))\to t(\sigma(a))= t(g)= x$. Hence, after passing to a subsequence, we may assume that $u_n$ converges to some $u\in G$.
Since $\sigma(a_n)\to g$, it follows that $(g_n)_{n\geq 1}$ is convergent. That implies that $a_n= t(g_n)$ is convergent in the topology of $\mathcal{O}_x$ (and the limit is clearly $a$).
For (2), one first chooses a transversal $T\subset M$ to $\OO_x$ through $x$ (of complementary dimension). For $y\in T$, the fact that $T$ is transversal to $\mathcal{O}_y$ at $y$,
$T_y(T)+ T_y(\mathcal{O}_y)= T_yM$, can be re-written as $T_y(T)+ \rho(A_y)$, where again we use the Lie algebroid $A$ of $\G$. In particular, the
transversality condition is open in $y$ hence, after passing to a neighborhood $T'$ of $x$ in $T$, it follows that the transversality condition will be satisfied at all points in $T'$. Finally, since
$T'$ is transversal to $\mathcal{O}_x$ at all the intersection points, and are of complementary dimensions in $M$, their intersection will be a $0$-dimensional submanifold, hence discrete.
Hence we can choose a neighborhood $\Sigma$ of $x$ in $T'$ such that $\Sigma\cap \mathcal{O}_x= \{x\}$; $\Sigma$ has all the desired properties. 
\end{proof}


\subsection{End of proof of Theorem \ref{main-theorem}}  
Combine the previous proposition with the Morita equivalence from Lemma \ref{lemma-transversals}, the Morita invariance of linearizability (Proposition \ref{Morita-equivalence}) and the fixed point case (Theorem \ref{th-Zung}).



\section{Variations on the linearizability; the proofs of the corollaries}
\label{Variations}

In this section we discuss some possible variations on the linearization theorem, in which 
invariance of the neighborhoods are required (the invariant and semi-invariant linearization mentioned already
in Definition \ref{versions-linearization}). Of course, this requires variations on the hypothesis of the
theorem as well. In particular, we will prove the corollaries from the
introduction. However, the main conclusion of this subsection is the following: once linearizability is achieved
(as in Theorem \ref{main-theorem}), the passing to stronger versions is just a problem about equivariant bundles.


\subsection{\underline{Variations on the hypothesis}}
\label{var-on-hypothesis}
One of the variations of the main theorem  arises by imposing conditions on the topology of the orbit. 
(think e.g. that Reeb stability is known to hold also for non-compact leaves with ``special topology''). 
Relevant for Corollary \ref{cor1} is the notion of finite type manifolds. As we have already mentioned in the
introduction, the only property of finite type manifolds that we will use is the 
following generalization of Ehresmann's theorem.

\begin{lemma}\label{Weinstein-lemma} (Weinstein semi-stability \cite{Wein3}) If $\alpha: X\rmap Y$ is a submersion, and $y\in Y$ is such that $X_0= \alpha^{-1}(0)$ is a finite type manifold
then, on  a neighborhood $\mathcal{U}$ of $X_0$, $\alpha$ is
equivalent to the product fibration $\textrm{pr}_2: X_0\times \alpha(\mathcal{U})\rmap \alpha(\mathcal{U})$.

In the $K$-equivariant case, where $K$ is a compact Lie group, if we replace the condition that $X_0$ is of finite type
by the condition that the action of $K$ on $X_0$ is free and $X_0/K$ is of finite type, then the equivalence can be chosen $K$-invariant.
\end{lemma}

The other variations of the main theorem arise by imposing stronger properness conditions. The main such condition is that of $s$-properness, which was already mentioned in 
Definition \ref{def-properness}. Of course, $s$-properness at $x$ implies compactness of the orbit through $x$. However, the main difference
between the proper and the $s$-proper case is the (possible) behaviour of $\G$ around the orbit. 

\begin{definition} An orbit $\OO$ of a Lie groupoid $\G$ is called \textbf{(topologically) stable} (for $\G$) if it admits arbitrarily small invariant neighborhoods.
\end{definition}

\begin{remark} \label{remark-stability}\rm \ Assume that $\G$ is proper at $x$. Then, in general, $\OO_x$ is not stable
(think e.g. of proper actions of non-compact Lie groups). Actually, 
stability of $\OO_x$ implies that $\OO_x$ must be compact, and then $\G$ will be $s$-proper at $x$. 
This follows e.g. from the linearization theorem and then by looking at the local model (see Proposition \ref{explanation-prop-cor} below). 

Still under the assumption that $\G$ is proper at $x$, note that even if $\OO_x$ is compact, it may still fail to be stable
(even for groupoids associated to submersions, see e.g. Example 3.4 in \cite{Wein3}). 
In particular, compactness of $\OO_x$ alone does not imply $s$-properness at $x$, 
unless we make extra assumptions (e.g. that $\OO_x$ is stable or that the $s$-fibers of $\G$ are connected).

The following lemma shows that, in contrast, $s$-properness at $x$ implies not only that $\OO_x$ is compact, but also stable. 
\end{remark}

\begin{lemma}\label{s-prop-stable} If the Lie groupoid $\G$  is $s$-proper at $x$,  then
$\OO = \OO_x$ admits arbitrarily small invariant neighborhoods over which $s$ becomes a locally trivial fibration.
In particular, $\OO$ is stable for $\G$.
\end{lemma}

\begin{proof} We first prove that $\OO$ is stable. Assume it is not. Then 
it admits a neighborhood $U$ such that the $\G$-saturation of any neighborhood of $x$ is not inside $U$.
Hence we find a sequence of $x_n$ converging to $x$, such that for each $n$ there is an arrow $g_n: x_n\rmap y_n$ with $y_n\notin U$. 
Since $s(g_n)= x_n$ converges to $x$, using the $s$-properness at $x$, we may assume that $(g_n)$ is convergent. But then $y_n= t(g_n)$ will converge to 
a point $y\in \mathcal{O}\subset U$. This is in contradiction with $y_n\notin U$.

Now, let $V$ be an arbitrarily small neighborhood of $\OO$. By the first part, we may assume that it is $\G$-invariant.
Since $s$ is locally trivial around $x$, we find a neighborhood $U_0 \subset V$ of $x$ such that, restricted to $s^{-1}(U_0)$, $s$ is a trivial
fibration. Let $U$ be the saturation of $U_0$. It is invariant and contained in $V$. We claim that $s$ is locally trivial over $U$. Let $y\in U$ be arbitrary.
We find $g: x\rmap y$ for some $x\in U_0$. We choose a section $\sigma: U_1\rmap \G$ of the source map defined on some open neighborhood
$U_1$ of $x$, such that $\sigma(x)= g$ and $\phi_{\sigma}:= t\circ \sigma: U_1\rmap V_1$ is a diffeomorphism onto an open neighborhood
$V_1$ of $y$. We may assume $U_1\subset U_0$ and $V_1\subset U$. It follows that $s$ is trivial over $V_1$ since 
$a\mapsto a\sigma(s(a))^{-1}$ is an isomorphism
from the bundle $s: s^{-1}(U_1)\rmap U_1$ to $s: s^{-1}(V_1)\rmap V_1$, covering $\phi_{\sigma}$. 
\end{proof}


\subsection{\underline{The extended local model}}
\label{The new local model}

In this subsection we explain that, although the linearization of $\G$ at $\OO$ 
gives a priori a (linear) local model for $\G$ only on some neighborhood of $\OO$, 
one can actually extend it to a (non-linear) local model for $\G$ on an {\it invariant} neighborhood of $\OO$. 
We call it the extended local model; it is built out of  a certain $G_x$-equivariant bundle over $\NN_x$
and the linear model corresponds to the trivial bundle.

The extended local model comes out as the maximal information that the Morita equivalence point of view
gives us in the linearizable case. Assume that $H$ is a Lie group acting linearly on a vector space
$\mathcal{N}$. To illustrate the relevant structure, we introduce the category $\mathcal{B}un_{H}(\NN, 0)$ 
which has 
\begin{itemize}
\item as objects: manifolds $P$ endowed with a free and proper action of $H$ and an $H$-equivariant submersion $\alpha: P\rmap \NN$
whose image contains the origin.
\item morphisms: $H$-equivariant maps compatible with the $\alpha$'s.
\end{itemize}
We think of the objects as $H$-equivariant bundles over a neighborhood of the origin in $\NN$ (not necessarily locally trivial!).
For such an object $P$, we consider 
\[ U_{P}:= P/H, \  P_0= \alpha^{-1}(0), \  \OO_{P}:= P_0/H\subset U_{P} .\]
Using Example \ref{example-M-isotropy}, we associate to any object $P\in \mathcal{B}un_{H}(\NN, 0)$
its gauge groupoid 
\[
\xymatrix{
\textrm{Gauge}(P):= \G(\alpha)/H= (P\times_{\alpha} P)/H \ar[r] \ar@<+1ex>[r]&   U_{P}= P/H,  }
\]
the quotient of the groupoid $\G(\al)$ associated to $\alpha$ (Example \ref{ex-submersion}) modulo the
(diagonal) action of $H$. When we need to be more precise, we use the notation $\textrm{Gauge}(P, \alpha)$ or
$\textrm{Gauge}_{H}(P, \alpha)$. For instance, in the case of an orbit $\OO= \OO_x$ of a Lie groupoid $\G$, 
if we apply this construction to the linear object 
\begin{equation} 
\label{linear-object}
(P_x\times \NN_x, \textrm{pr}_2) \in \mathcal{B}un_{G_x}(\NN_x, 0),
\end{equation}
we obtain the local model $\NN_{\OO}(\G)$ in its bundle description. 

\begin{remark}\rm \  Note that for any $P\in \mathcal{B}un_{H}(\NN, 0)$, 
$\OO_P$ is an orbit of $\textrm{Gauge}(P)$ and the resulting isotropy data is
$H$ (as isotropy group), $\NN$ (as the isotropy representation) and $P_0$ (as the isotropy bundle).
Hence the linearization of $\textrm{Gauge}(P)$ at $\OO_P$ is precisely the gauge groupoid of
the linear object 
\[ (P_0\times \NN, \textrm{pr}_2)\in \mathcal{B}un_{H}(\NN, 0) .\]
Also, by construction, $\textrm{Gauge}(P)$ is Morita equivalent to the groupoid associated to the action of $H$ on $\alpha(P)$
In particular, $\textrm{Gauge}(P)$ is linearizable at $\OO_{P}$.
\end{remark}

\begin{remark} \rm \ Since we will be interested in understanding $\textrm{Gauge}(P)$ around neighborhoods of $\OO_P$
in $M_P$, let us point out:
\begin{enumerate}
\item Open neighborhoods of $\OO_P$ in $U_P$ correspond to \textbf{open sub-objects} of $P$, i.e. 
$H$-invariant open sets $Q\subset P$ containing $P_0$. The neighborhood associated to $Q$ is $U_Q\subset U_P$.
Moreover,
\[ \textrm{Gauge}(P)|_{U_Q}= \textrm{Gauge}(Q).\]
\item Open invariant neighborhoods of $\OO_P$ in $U_P$ correspond to \textbf{open sub-bundles} of $P$, 
i.e. $Q\subset P$ of type $\alpha^{-1}(W)$ where $W$ is an $H$-invariant neighborhood of the origin, contained in
$\alpha(P)$. 
\end{enumerate}
\end{remark}

With this, the main result of this subsection is the following.

\begin{proposition} \label{new-local-model} If $\G$ is linearizable at $\OO= \OO_x$, then there exists an invariant neighborhood 
$U$ of $\OO$, and $P\in \mathcal{B}un_{G_x}(\NN_x)$ such that
\[ \G|_{U}\cong \textrm{Gauge}(P) \]
by a diffeomorphism which sends $\OO$ to $\OO_P$. Moreover, if $\OO$ is a stable orbit, then $P$ may be chosen to be an open sub-object of (\ref{linear-object}). 
\end{proposition}

\begin{proof} Consider $U$ and $V$ such that $\G|_{U}$ is isomorphic to $\NN_{\OO}(\G)|_{V}$. Passing to their
saturation (see Example \ref{example-saturation}), we obtain invariant neighborhoods $U$ and $V$ such that $\G|_{U}$ is Morita equivalent to $\NN_{\OO}(\G)|_{V}$.
In turn, the last groupoid is Morita equivalent to the action groupoid $G_x\ltimes V_x$  where $V_x= V\cap \NN_x$ (see Remark \ref{example-M-isotropy}).
Then a bi-bundle $P$ which realizes a Morita equivalence between $G_x\ltimes V_x$ and $\G|_{U}$ is precisely an object $P\in \mathcal{B}un_{G_x}(\NN_x)$
such that $\G|_{U}\cong \textrm{Gauge}(P)$ (see Example \ref{example-M-isotropy}). Of course, this bundle $P$ is the same as the one that appears in
the proof of Proposition \ref{Morita-equivalence}. The argument given there shows that $P$ admits an open sub-object $P'$ which can be openly embedded as a sub-object 
of $P_x\times \NN_x$. On the other hand, the stability of $\OO$ for $\textrm{Gauge}(P)$ (hence also for $\G$) is equivalent to the fact that any 
open sub-object of $P$ contains a smaller open sub-bundle. Hence $P'$ can be chosen so that it is a sub-bundle of $P$. That means that the
open set $U'$ that it induces on $M$ is invariant and $\G|_{U'}= \textrm{Gauge}(P')$.
\end{proof}

\subsection{\underline{Semi-invariant and invariant linearizability}}
Here we deduce the following improvements of Corollary \ref{cor1} and \ref{cor2} from the introduction.

\begin{proposition}\label{prop-cor1} If $\G$ is linearizable at $\OO= \OO_x$, $G_x$ is compact and $\OO$ is of finite type, then
there are arbitrarily small neighborhoods $U$ of $\OO$ such that $\G|_{U}\cong \NN_{\OO}(\G)$.
\end{proposition}

\begin{proposition}\label{prop-cor2} If $\G$ is linearizable at $\OO= \OO_x$, $G_x$ is compact and $\OO$ is a stable 
orbit of $\G$, then there are arbitrarily small neighborhoods $U$ of $\OO$ such that $\G|_{U}\cong \NN_{\OO}(\G)$.
\end{proposition}

To clarify the hypothesis, let us also point out:

\begin{proposition}\label{explanation-prop-cor} If $\G$ is linearizable at $\OO= \OO_x$, then
\begin{enumerate}
\item[(a)] $\G$ is proper at $x$ if and only if $G_x$ is compact. 
\item[(b)] $\G$ is $s$-proper at $x$ if and only if $G_x$ is compact and $\OO$ is a stable orbit. 
\end{enumerate}
\end{proposition}

For the proofs, we need the following.

\begin{lemma} Let $P\in \mathcal{B}un_{\HH}(\NN)$, $\G= \textrm{Gauge}(P)$. Then $\G$ is proper at $x\in\OO_P$ if and only if $H$ is compact. 
Moreover, in this case, for any invariant neighborhood $U$ of $\OO$, there exists a smaller invariant neighborhood $U'$
such that $\G|_{U'}\cong \NN_{\OO}(\G)$.
\end{lemma}


\begin{proof} That $G_x$ must be compact is clear. Conversely, compactness of $G_x$ implies that $G_x\ltimes \NN_x$ is proper, hence
with Lemma \ref{cor-comp-lin} and Morita invariance of properness imply that $\G$ is proper at $x$. The last part follows from the fact that the origin $0\in \NN_x$ admits arbitrarily small $G_x$-invariant
neighborhoods which are equivariantly diffeomorphic to $\NN_x$.
\end{proof}

\begin{proof} (of Proposition \ref{explanation-prop-cor}) We are left with (b), in which the direct implication is part of Lemma \ref{s-prop-stable}.
Assume now that $\OO$ is stable and we show $s$-properness at $x$. Using again the previous lemma, we may assume $\G= \NN_{\OO}(\G)$.
Using the bundle picture for the linearization, we see that it suffices to prove that $\OO$ is compact. Assume it is not. Then
we find a smooth function $\epsilon: \OO\rmap (0, \infty)$ whose infimum over $\OO$ is $0$. Choose also a $G_x$-invariant metric on $\NN_x$ and define 
\[ P_{\epsilon}= \{(p, v)\in P_x\times \NN_x: ||v||< \epsilon(\pi(p))\} \]
where $\pi: P_x\rmap \OO$ is the projection. Then $P_{\epsilon}/G_x$  defines an open neighborhood $U_{\epsilon}$ of $\OO$. Since we can find a smaller invariant neighborhood inside of $U_{\epsilon}$,
we can find an open sub-bundle $P'\subset P_{\epsilon}$, i.e. a neighborhood $W$ of the origin in $\NN_x$ such that $P_x\times W\subset P_{\epsilon}$. But then, using any 
non-zero $w_0\in W$, we have $\epsilon > ||w_0||$ on $\OO$, which is in contradiction with the choice of $\epsilon$. 
\end{proof}

\begin{proof} (of Proposition \ref{prop-cor1} and of Proposition \ref{prop-cor2})  The second proposition follows from the last lemma. For the first one
we may assume that $\G= \textrm{Gauge}(P)$ and we have to show that, for any open sub-object $P'\subset P$, there is a smaller one which is isomorphic to the product of $P_x$ with a neighborhood of the origin.
This is Weinstein's Lemma \ref{Weinstein-lemma} for $\alpha: P'\rmap \NN_x$, $K= G_x$ (so that $\alpha^{-1}(0)/K= \OO$ is
of finite type). When $\OO$ is stable, the last lemma clearly implies Proposition \ref{prop-cor2}.
\end{proof}

\subsection{\underline{The open problem}}
It is clear that the main theorem and its corollaries cover a large number of related ``classical results'' from group actions and foliations.
However, there is one intriguing situation that falls out: the case of proper actions of non-compact Lie groups, case in which we do know that
invariant linearizability does hold. The associated groupoids are, of course, not $s$-proper and the orbits are not stable (hence the
invariant linearization holds, but not on arbitrarily small invariant neighborhoods). 
As pointed out by Weinstein (see Example 10.1 in \cite{Wein3})), it is tempting to think that the local triviality of the source map $s$ should
play a role. At this point, the condition that looks most natural to us is the following.

\begin{definition}\label{def-inv-triviality} Let $\G$ be a Lie groupoid over $M$. 
We say that \textbf{$s$ is inv-trivial around $x\in M$} if there exists an invariant neighborhood $U$ of $x$ so that,
above $U$, 
\[ s|_{s^{-1}(U)}: s^{-1}(U)\rmap U\]
is equivalent to a product fibration (i.e. to $\textrm{pr}_2: s^{-1}(x)\times U\rmap U$).
\end{definition}

Of course, action groupoids satisfy this condition.

\begin{remark}\label{remark-s-triviality}\rm\ 
The condition of {\it $s$-properness at $x$} implies that $s$ is trivial around $x$, and then that $s$ is locally trivial on 
arbitrarily small invariant neighborhoods of $x$. However, this does not imply the {\it inv-triviality of $s$ around $x$}. One difference
takes place along the orbits -- think e.g. about compact principal bundles and the associated gauge groupoids (Example \ref{ex-principal-bundles}).

However, there is another difference between the two, which takes place transversally:
inv-triviality of $s$ around $x$ does not imply stability of  $\OO_x$
even if we further assume that $\G$ is proper at $x$ (i.e. the specified hypothesis of our open problem) or that $\OO_x$ is compact
(see also Remark \ref{remark-stability}). Think e.g.
about proper actions (first case) or about compact orbits (second case) for actions of non-compact Lie groups
(Example \ref{ex-action}). Note that these examples are of completely different nature than the transitive ones 
mentioned above (in which the orbit is clearly stable). 
\end{remark}

We leave to the reader the puzzling exercise of translating our problem into one about
equivariant bundles. When wondering about the uniqueness of the extended local model (i.e. of $P$ in Proposition \ref{new-local-model}),
one should be aware of the Picard groups \cite{BW} of the action groupoids involved. The exercise will eventually bring the reader to
rigidity questions similar to the ones of Palais \cite{Palais, Palais2}, but for non-compact manifolds
endowed with free actions, rising also the question of center of mass-techniques \cite{Grove} for
the linearization of proper groupoids (question that already appears in \cite{Wein2}).

\bibliographystyle{amsplain}

\begin{thebibliography}{11}

\bibitem{BW} H.~Bursztyn and A.~Weinstein, Picard groups in Poisson geometry, 
\emph{Mosc. Math. J.},  \textbf{4} (2004), 39–66.

\bibitem{Conn2} J.~Conn, Normal forms for smooth Poisson structures,
  \emph{Annals of Math.}~\textbf{121} (1985), 565--593.



\bibitem{Crainic} M.~Crainic, Differentiable and algebroid cohomology, van Est isomorphisms, and characteristic classes, \emph{Comment.~Math.~Helv.~}\textbf{78} (2003), no. 4, 681--721.

\bibitem{Nuno} M.~Crainic, J. N.~Mestre and I.~Struchiner, Deformations of Lie groupoids, work in progress.

\bibitem{CrFe} M.~Crainic and R.~L.~Fernandes, A geometric approach to Conn's Linearization Theorem, \emph{Annals of Math.}, to appear.


\bibitem{Ionut} M.~Crainic and I.~M\v{a}rcu\c{t}, \emph{A Normal Form Theorem around Symplectic Leaves}, preprint arXiv:1009.2090 

\bibitem{DuZu} J.P.~Dufour and N.T.~Zung, \emph{Poisson structures and their normal forms}, Progress in Mathematics, 242. Birkh$\ddot{a}$user Verlag, Basel, 2005.


\bibitem{DuKo} J.~J.~Duistermaat and J.~Kolk, \emph{Lie groups}, Universitext, Springer-Verlag Berlin, 2000.

\bibitem{Grove}K.~Grove and H.~Karcher, How to conjugate C1-close group actions. \emph{Math.~Z.~} \textbf{132} (1973), 11Ð20. 

\bibitem{Haefliger} A.~Haefliger, Groupoids and foliations, in Groupoids in analysis, geometry, and physics (Boulder, CO, 1999),  83–100, Contemp. Math., 282, 2001.

\bibitem{Lang} S.~Lang, \emph{Differential and Riemannian manifolds}, Graduate Texts in Mathematics, 160. Springer-Verlag, New York, 1995.



\bibitem{MkXu} K.~Mackenzie and P.~Xu, Classical lifting processes and multiplicative vector fields, 
\emph{Quart.~J.~Math.~Oxford},  \textbf{49} (1998), 59–-85.

\bibitem{MoMr} I.~Moerdijk and J.~Mrcun, \emph{Introduction to foliations and Lie groupoids}, Cambridge Studies in Advanced Mathematics, 91. Cambridge University Press, 2003.

\bibitem{Palais} R.~Palais and T.~Stewart, Deformations of compact differentiable transformation groups, 
\emph{Amer. J. Math.},  \textbf{82} (1960), 935-937.

\bibitem{Palais2}R.~S.~Palais, Equivalence of nearby differentiable actions of a compact group. \emph{Bull. Amer. Math. Soc.} \textbf{67} (1961) 362Ð364. 




\bibitem{Hessel} H.~Posthuma, M.~Pflaum and X.~Tang, \emph{Geometry of orbit spaces of proper Lie groupoids}, preprint arXiv:1101.0180.

\bibitem{Trent} G.~Trentinaglia, Tannaka duality for proper Lie groupoids, \emph{Journal of Pure and Applied Algebra} \textbf{214} (2010), no.6, 750--768.

\bibitem{Wein2} A.~Weinstein, Linearization problems for Lie algebroids and Lie groupoids, \emph{Lett.~Math. Phys.~}\textbf{52} (2000), 93--102.

\bibitem{Wein3} A.~Weinstein, Linearization of regular proper groupoids, \emph{J.~Inst.~Math.~Jussieu} \textbf{1} (2002), no.3, 493--511.

\bibitem{Zung} N.T.~Zung, Proper Groupoids and Momentum Maps: Linearization, Affinity and Convexity,  \emph{Ann. Sci. ƒcole Norm. Sup.} \textbf{39} (2006), no. 5, 841--869.

\end{thebibliography}
\def\lllll{}

\end{document}